\documentclass[conference]{ieeeconf}
\IEEEoverridecommandlockouts

\usepackage{cite}
\usepackage{amsthm,amsmath,amssymb,amsfonts}
\usepackage{graphicx}
\usepackage{textcomp}
\usepackage{xcolor}
\usepackage{hyperref}

\usepackage{silence}
\ActivateWarningFilters[captions]
\usepackage{subcaption}
\DeactivateWarningFilters[captions]

\usepackage{balance}
\hypersetup{hidelinks}
\usepackage{mathtools}
\mathtoolsset{showonlyrefs}
\usepackage[normalem]{ulem}
\usepackage{fontawesome5}

\usepackage{algorithm}
\usepackage{algorithmicx}
\usepackage{algcompatible}
\usepackage{tikz}
\usepackage{pgfplots}
\usepgfplotslibrary{groupplots}
\pgfplotsset{compat=1.18}

\newlength\figW
\newlength\figH

\newtheorem{definition}{Definition}
\newtheorem{lemma}{Lemma}
\newtheorem{theorem}{Theorem}

\newtheorem{mechanism}{Mechanism}
\newtheorem{problem}{Problem}

\newcommand{\ER}{Erd\H{o}s-R\'{e}nyi }
\newcommand{\FN}{Fr\'echet-Nikodym }
\newcommand{\G}{\mathcal{G}}
\newcommand{\M}{\mathcal{M}}
\newcommand{\Prob}{\mathbb{P}}
\newcommand{\Adj}{\mathrm{Adj}}
\newcommand{\uti}{d_\Delta}

\def\BibTeX{{\rm B\kern-.05em{\sc i\kern-.025em b}\kern-.08em
    T\kern-.1667em\lower.7ex\hbox{E}\kern-.125emX}}

\title{Generating Differentially Private Networks \\ with a Modified
\ER Model} 

\author{Huaiyuan Rao, Calvin Hawkins, Alexander Benvenuti, Matthew Hale\textsuperscript{1}\thanks{\textsuperscript{1}School of Electrical and Computer Engineering, Georgia Institute of Technology, Atlanta, GA, USA.
Emails: {\texttt{\{hrao43, chawkins64, abenvenuti3, matthale\}@gatech.edu}}}
\thanks{
This work was supported by 
NSF under CAREER grant 2422260 and Graduate Research Fellowship grant DGE-2039655,
AFOSR under grant FA9550-19-1-0169, and
ONR under grant N00014-21-1-2502.}}

\begin{document}
\maketitle
\begin{abstract}
    Differential privacy has been used to privately calculate numerous network properties,
    but existing approaches often require the development of a new privacy mechanism 
    for each property of interest. 
    Therefore, we present a framework for generating entire networks in a differentially private way. 
    Differential privacy is immune to post-processing, which allows for \emph{any} 
    network property to be computed and analyzed for a private output network, 
    without weakening its protections. 
    We consider undirected networks and develop a differential privacy
    mechanism that takes in a sensitive network and outputs a private network by randomizing its edge set.    
    We prove that this mechanism does provide differential privacy to a network's edge set, 
    though it induces a complex distribution over the space of output graphs.     
    We then develop an equivalent 
    privacy implementation using a modified Erd\H{o}s-R{\'e}nyi model
     that 
    constructs an output graph edge by edge, and it
    is efficient and easily implementable, 
    even on large complex networks. 
    Experiments implement~$\varepsilon$-differential privacy with $\varepsilon=2.5$ when computing graph Laplacian spectra, and these results show the proposed mechanism incurs $49.34\%$ less error than the current state of the art.  
\end{abstract}

\section{Introduction}
Graphs are used to model complex networks in applications including social systems~\cite{aggarwal2011introduction, tabassum2018social}, epidemics~\cite{pare2020modeling, chen2025scalable}, and multi-agent control systems~\cite{shi2020survey, qin2016recent}.
Graphs often use nodes to model individuals and edges to model connections
between them, e.g., cohabitation. 
Graphs can therefore be sensitive 
because their edges can encode one's personal relationships and other sensitive information. 
Simultaneously, there is often a need to analyze graphs for various properties,
such as the counts of certain subgraphs, and releasing this information may inadvertently
reveal sensitive details about a graph.


Indeed, while it is perhaps surprising, it is well-established that even scalar-valued
graph properties can reveal information about a graph, including information
about individuals~\cite{ding2018privacy, sun2019analyzing, day2016publishing, chen2021edge, hawkins2023node, hay2009accurate, gupta2012iterative, wang2013learning, nguyen2016detecting, proserpio2014calibrating}.
The risk of leaking sensitive information is only increased if multiple graph properties are released, e.g.,
a graph's degree sequence and counts of certain subgraphs it contains.
What results is a need to analyze and share graph properties, alongside a need
to conceal information about individuals that is contained within those graphs. 


Differential privacy has been used to privatize a wide array of graph analyses of this kind.
Broadly, differential privacy ensures that outputs of database queries are insensitive to the addition, removal, 
and/or change
of any single individual's information in a database. 
This formalism ensures information about individuals is provably
hard to infer from privatized data. 
By treating a graph as a database, 
differential privacy has been applied
to computing graph properties including subgraph counts~\cite{sun2019analyzing, ding2018privacy, blocki2013differentially, kasiviswanathan2013analyzing}, degree distributions~\cite{day2016publishing}, the spectra of graph Laplacians~\cite{chen21,hawkins2023node, chen2024bounded}, clustering coefficients~\cite{wang2013learning}, and cut queries~\cite{gupta2012iterative}.
These mechanisms all enable the release of privatized graph properties while provably protecting individual agents' connections
in a network. 

However, those approaches have required the development of a new privacy mechanism for each of those properties, which
has required privacy engineers to anticipate which graph properties will be privately
computed in the future. 
There are various quantitative graph properties that do not have a purpose-built privacy
mechanism, and it is unclear how one could privately compute those properties. 

Therefore, in this paper we develop a new privacy mechanism that generates an entire graph in a differentially
private way. We use edge differential privacy, which privatizes the edge set of a graph, and
we first develop a privacy mechanism
over the space of graphs using
the exponential mechanism~\cite{mcsherry2007mechanism}. 
Then we derive a closed form for the probability distribution 
this approach induces over the space of graphs, though it is complex
and difficult to  directly sample from. 

Therefore, we next design 
a modified \ER (E-R) model to generate random output graphs edge by edge. 
The conventional E-R model~\cite{erdds1959random} includes each edge with the same probability,
and ours is ``modified'' because it divides edges into two types,  
and edges of different types appear with different probabilities. 
We derive edge probabilities for this E-R model that make it equivalent to the exponential
mechanism, which provides privacy with much lower computational complexity. 
Because differential privacy is immune to post-processing, \emph{any} graph properties
can be computed with a private output graph, without weakening privacy's protections
and without requiring a purpose-built mechanism for each one.

To summarize, our contributions are threefold: 
\begin{enumerate}
    \item We design the graph exponential mechanism, which synthesizes an entire differentially private graph (Mechanism~\ref{mech:em}, Theorem~\ref{thm:mech1_dp}).
    \item We develop a modified E-R model, which provides a computationally efficient implementation of the exponential mechanism (Mechanism~\ref{mech:2}, Theorem~\ref{thm:equiv}).
    \item We show in simulation that our mechanism attains~$49.34$\% less error than a state-of-the-art mechanism when computing graph Laplacian spectra (Section~\ref{sec:IV}). 
\end{enumerate}

We point out that~\cite{sala2011sharing, mir2012differentially,qin2017generating} 
also generate private output graphs, but we fundamentally differ from them.
Those works compute certain graph properties (e.g., degree sequences), privatize
their values, and then synthesize graphs whose properties (exactly or approximately) 
take on those values. The accuracy of output graphs
is biased in favor of the chosen graph properties, but 
the accuracy of this approach is unclear for 
other unrelated graph properties. 

In contrast, our approach directly synthesizes an output graph from a sensitive
input graph and does not bias its accuracy toward any particular property. 
We show that it outperforms the state of the art by $49.34$\% when computing
graph Laplacian spectra, which illustrates the success of our approach. 
Prior work has applied differential privacy to trajectory-valued 
data~\cite{le2013differentially,yazdani20},
various properties of Markov processes~\cite{uaipaper,benvenuti2024differentially}, and
other problems in control. 
To the best of our knowledge,
this work is the first to generate and analyze entire private
graphs for control-oriented analysis, namely computing
graph Laplacian spectra. 

The rest of the paper is organized as follows. Section~\ref{sec:II} provides background and problem statements. Section~\ref{sec:III} develops the graph exponential mechanism to generate privacy-preserving graphs.
Section~\ref{sec:eff_mech} derives an efficient implementation of the graph exponential mechanism
using a modified E-R model.
Section~\ref{sec:IV} compares the accuracy of the proposed mechanism to the current state of art for privatizing graph Laplacian spectra, and Section~\ref{sec:V} concludes.

\textit{Notation:} We use $\mathbb{R}$ to denote the real numbers, $\mathbb{Z}$
to denote the integers, and $\mathbb{N}$ to denote the positive integers. 
For $n\in \mathbb{N},$ we use $[n]$ to denote the set $\{1,2, \dots, n\}$. We use $|S|$ to denote the cardinality of a finite set $S$, and we use $S_1\Delta S_2=(S_1\backslash S_2) \cup (S_2\backslash S_1)$ to denote the symmetric difference of sets $S_1$ and $S_2$. For $n\in \mathbb{N}$, we use $\G_n$ to denote the set of 
simple, undirected, unweighted graphs on $n$ nodes. We use $\lfloor \cdot \rfloor$ to denote the floor function.

\section{Background and Problem Formulation} \label{sec:II}
In this section, we review graph theory and differential privacy, and then we provide formal problem statements.

\subsection{Graph Theory}
We consider a simple, undirected, unweighted graph $G=(V,E)$ defined over a set of nodes $V=[n]$ with edge set $E \subset V \times V$. We say $(i,j) \in E$ if nodes $i$ and $j$ share an edge and $(i,j) \notin E$ if they do not.
Given $G=(V,E)$, we denote the edge set by $E(G)$ when there is a need to identify the underlying graph~$G$. 
We denote the complementary graph of $G=(V,E)$ as $\bar{G}=(V,\bar E)$, where~$\bar E=\{(i,j)\in V\times V\ |\ (i,j)\notin E\}$
is all edges that are not in~$G$. 
 
\subsection{Differential Privacy}
We use differential privacy~\cite{dwork2014algorithmic} to provide privacy for graphs. In this work, we consider the edge set of a graph to be the sensitive data that must be protected, and we implement \emph{edge differential privacy}~\cite{hay2009accurate}
to protect it. The goal of differential privacy is to 
randomize data in a way that
makes two ``similar" pieces of sensitive data produce outputs that are ``approximately indistinguishable". The notion of ``similarity" is formalized by an adjacency relation, which specifies 
which pairs of sensitive data must be rendered approximately indistinguishable by differential privacy. 
Next, we define a metric that describes the closeness of graphs, which we will use to define adjacency for edge sets.

\begin{definition}[\FN Metric]\label{def:FN}
    For two finite sets $S_1$ and $S_2$, the \FN metric with respect to the counting measure is
        $\uti(S_1, S_2)=|S_1 \Delta S_2|$.
\end{definition}

In words, $d_{\Delta}(S_1, S_2)$ counts the total number of elements
that are in either~$S_1$ or~$S_2$ but not in both of them. 
With an abuse of terminology, we refer to $\uti$ as \emph{the} \FN distance. We define adjacency for graphs
using Definition~\ref{def:FN}. 

\begin{definition}[{Edge adjacency; \cite{hay2009accurate}}]\label{def:adj}
    Fix an adjacency parameter~$A\in\mathbb{N}.$ Two graphs $G=(V,E)$ and $G' =(V,E')$ are adjacent if
    \begin{equation}
    \uti(E,E')\leq A.
    \end{equation}
    We write~$\Adj(G,G') = 1$ if~$G$ and~$G'$ are adjacent and~$\Adj(G, G')=0$ otherwise.
\end{definition}

Definition~\ref{def:adj} says that two graphs~$G$ and~$G'$ are adjacent if they 
have the same number of nodes and 
differ in at most~$A$ edges. An example of adjacent graphs is shown in Figure~\ref{fig:adj}.
The notion of ``approximately indistinguishable'' is made precise by the definition of differential privacy itself.

\begin{figure}[!t]
    \centering

\definecolor{pastelBlue}{RGB}{166,206,227}
\definecolor{pastelGreen}{RGB}{178,223,138}
\definecolor{pastelPink}{RGB}{251,180,174}
\definecolor{pastelPurple}{RGB}{202,178,214}

\begin{tikzpicture}[
    scale=0.6,
    node/.style={circle, draw=black!60, line width=0.5pt,
                 inner sep=0pt, minimum size=6.5mm, font=\footnotesize},
    edge/.style={- , draw=black!70, very thick}
]

\def\dr{1.4cm} 

\node[node, fill=pastelBlue]   (a1) at (0,\dr)    {};
\node[node, fill=pastelPink]   (a2) at (\dr,0)    {};
\node[node, fill=pastelGreen]  (a3) at (0,-\dr)   {};
\node[node, fill=pastelPurple] (a4) at (-\dr,0)   {};

\draw[edge] (a1)--(a2);
\draw[edge] (a2)--(a3);
\draw[edge] (a3)--(a4);
\draw[edge] (a4)--(a1);
\draw[edge] (a1)--(a3); 
\draw[edge] (a2)--(a4); 

\node at (0,-2.3) {Original Graph};

\begin{scope}[xshift=5cm]
\node[node, fill=pastelBlue]   (b1) at (0,\dr)    {};
\node[node, fill=pastelPink]   (b2) at (\dr,0)    {};
\node[node, fill=pastelGreen]  (b3) at (0,-\dr)   {};
\node[node, fill=pastelPurple] (b4) at (-\dr,0)   {};

\draw[edge] (b1)--(b2);
\draw[edge] (b2)--(b3);
\draw[edge] (b3)--(b4);
\draw[edge] (b4)--(b1);

\node at (0,-2.3) {Adjacent graph};
\end{scope}

\begin{scope}[xshift=10cm]
\node[node, fill=pastelBlue]   (c1) at (0,\dr)    {};
\node[node, fill=pastelPink]   (c2) at (\dr,0)    {};
\node[node, fill=pastelGreen]  (c3) at (0,-\dr)   {};
\node[node, fill=pastelPurple] (c4) at (-\dr,0)   {};

\draw[edge] (c3)--(c2); 
\draw[edge] (c3)--(c4); 
\draw[edge] (c4)--(c1); 

\node at (0,-2.3) {Non-adjacent graph};
\end{scope}

\end{tikzpicture}
    \caption{An example of an adjacent and non-adjacent graph for $A=2$.
    The middle graph is adjacent to the left graph because it differs by~$2$ edges. The right graph is not adjacent to the left graph because
    it differs by~$3$ edges.} 
    \label{fig:adj}
\end{figure}
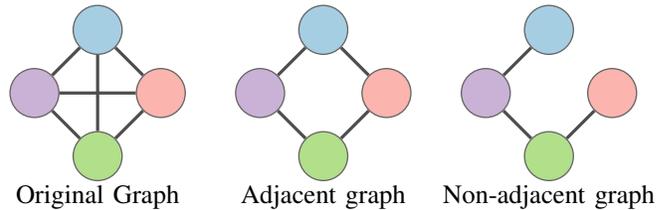

\begin{definition}[{Differential privacy; \cite{dwork2014algorithmic}}]\label{def:dp}
    Fix a probability space~$(\Omega, \mathcal{F}, \mathbb{P})$ and the set of all graphs with~$n$ nodes~$\G_n$. Fix a privacy parameter $\varepsilon \geq 0$. A randomized algorithm $\M:\G_n\times \Omega\to \G_n$ is $\varepsilon$-differentially private if for all $\mathcal{S} \subseteq \G_n$ and for all $G, G' \in \G_n$ such that $\Adj(G,G')=1$ we have
    \begin{equation}
        \mathbb{P}[\mathcal{M}(G) \in \mathcal{S}] \leq e^{\varepsilon} \cdot \mathbb{P}[\mathcal{M}(G')\in \mathcal{S}].
    \end{equation}
\end{definition}

Definition \ref{def:dp} says that 
if~$G$ and~$G'$ are adjacent, then
the probability of making an observation of the private form of $G$, namely $\M(G)$, must be close to the probability of making the same observation of the private form of $G'$, namely $\M(G')$. The parameter~$\varepsilon$ controls the strength of privacy, and smaller
values make privacy stronger. Typical values~\cite{hsu2014differential} are~$0.1$ to~$10$. 

A feature of differential privacy is that its protections weaken gracefully when multiple mechanisms are applied to the same piece of sensitive data. Repeated queries of sensitive data lead to a decreasing strength of privacy, which is formalized as follows.

\begin{lemma}[{Sequential composition; \cite[Section 3.5]{dwork2014algorithmic}}]\label{lemma:comp}
    Let $\M_i$ be an $\varepsilon_i$-differentially private mechanism for $i\in[k]$. If $\M_{[k]}$ is defined to be $\M_{[k]}=(\M_1, \M_2, \cdots, \M_k)$, then $\M_{[k]}$ is~$\sum_{i=1}^{k} \varepsilon_i$-differentially private.
\end{lemma}

Recalling that smaller values of~$\varepsilon$ imply stronger privacy, Lemma~\ref{lemma:comp} says that privacy weakens if multiple
mechanisms are used to generate private outputs from the same piece of sensitive data, but it weakens only linearly in the number of mechanisms. 
 Another appealing property of differential privacy is that it is immune to post-processing, i.e., arbitrary post-hoc computations on differentially private data do not weaken privacy's guarantees. This property is formalized by the following lemma.

\begin{lemma}[{Immunity to Post-Processing; \cite[Section 2.3]{dwork2014algorithmic}}]\label{lemma:immunity}
    Let~$\M:\mathbb{R}^n \times \Omega \rightarrow \mathbb{R}^m$ be an~$\varepsilon$-differentially private mechanism. Let~$h:\mathbb{R}^m \rightarrow \mathbb{R}^p$ be an arbitrary mapping. Then the composition~$h \circ \M:\mathbb{R}^n \to \mathbb{R}^p$ is $\varepsilon$-differentially private.
\end{lemma}

Lemma~\ref{lemma:immunity} implies that once we generate a private graph, we can compute and share any properties of it without weakening privacy's protections.

\subsection{Problem Statements}
The rest of the paper will solve the following problems. 

\begin{problem}\label{prob:1}
    Develop a mechanism to generate graphs that preserve $\varepsilon$-differential privacy in the sense of Definition~\ref{def:dp} for a sensitive graph~$G$.
\end{problem}

Problem~1 is only about privacy. It asks us to specify a randomized mechanism that maps a sensitive graph $G$ to a privatized 
graph $\widetilde G$ and enforces $\varepsilon$–differential privacy. 
Problem~2 then turns to accuracy. As stated in the introduction, our primary objective is to accurately privatize 
graphs, and in this work we assess that accuracy by examining
the Laplacian spectrum of a graph, which is commonly used in control 
and analysis of dynamical processes run over complex networks.

\begin{problem}\label{prob:3}
    Empirically compare the accuracy of the proposed mechanism to an existing benchmark mechanism on the Laplacian spectra of real-world social network data.
\end{problem}

 In short, Problem~\ref{prob:1} enforces differential privacy, and Problem~\ref{prob:3} measures how well the mechanism preserves
 the accuracy of Laplacian spectral properties.

\section{Differentially Private Graph Generation}\label{sec:III}
In this section, we design a mechanism for computing differentially private graphs from a sensitive graph, which solves Problem~\ref{prob:1}. 
Our approach is based on the exponential mechanism \cite{mcsherry2007mechanism}, which is 
commonly used to implement privacy for non-numerical data.
The core idea is to output private graphs whose edge sets are close to the edge set of a sensitive graph with high probability.
In Section~\ref{sec:utility}, we formalize this notion of closeness with a utility function, then in Section~\ref{sec:exp_mech} we use this utility function 
to develop the exponential mechanism we use.

\subsection{Utility and Global Sensitivity} \label{sec:utility}
For a given sensitive graph, we desire a privacy mechanism that outputs a private graph that is ``close" to the sensitive graph with high probability, while satisfying the definition of differential privacy. Since we consider edge differential privacy in this work, we consider graphs ``close'' when their edge sets are close, which we formalize with the following utility function. 

\begin{definition}[Utility function]\label{def:util}
    For a sensitive input graph~$G \in \G_n$, the private output graph~$H\in \G_n$ has utility
    \begin{equation}\label{eq:utilityy}
    u(G, H) = -\uti\big(E(G), E(H)\big).
    \end{equation}
\end{definition}

This choice of utility function captures the fact that a graph~$H$ is a better private
estimate of a sensitive graph~$G$ when the edge set of~$H$ differs from the edge set of~$G$ by fewer edges. 
The maximum value of the utility is~$0$, which occurs when~$H = G$, and all other choices
of private outputs give negative values of the utility function. 
The minimum value of the utility is~$-\binom{n}{2}$,
which occurs when~$E(H) = \bar{E}(G)$, i.e., when~$H$ is the complementary graph
to~$G$, which causes~$G$ and~$H$ to differ in all edges.

Implementing the exponential mechanism requires bounding the amount 
by which the utility can differ over adjacent graphs. This is referred to as the utility's \emph{global sensitivity}. 

\begin{definition}[{Global sensitivity; \cite[Section 3.4]{dwork2014algorithmic}}]\label{def:sensitivity}
    Fix~$A \in \mathbb{N}$ and consider the adjacency relation~$\textnormal{Adj}$ in Definition~\ref{def:adj}. 
    Then the global sensitivity of the utility function~$u: \G_n \times \G_n \to \mathbb{Z}$ from Definition~\ref{def:util} is 
    \begin{equation}
        \Delta u = \max_{H \in \G_n} \max_{\substack{G_1, G_2 \in \G_n \\ \textnormal{Adj}(G_1, G_2) = 1}} \big|u(G_1, H)-u(G_2, H)\big|. 
    \end{equation}
\end{definition}

The exponential mechanism is formulated in terms of this global sensitivity, which we bound
as follows. 

\begin{lemma} \label{lem:sensitivity}
    Fix an adjacency parameter~$A\in \mathbb{N}$. Then the global sensitivity from Definition~\ref{def:sensitivity} obeys~$\Delta u \leq A$. 
\end{lemma}

\begin{proof}
    First, suppose~$u(G_1,H)\geq u(G_2,H)$. 
    From Definition~\ref{def:sensitivity} we have 
    \begin{equation} \label{eq:lem3_step1}
        \Delta u = \max_{H \in \G_n} \  \max_{\substack{G_1, G_2 \in \G_n \\ \textnormal{Adj}(G_1, G_2) = 1}} u(G_1,H)-u(G_2,H), 
    \end{equation}
    and substituting the definition of~$u$ from~\eqref{eq:utilityy} gives
    \begin{multline}
        u(G_1, H)-u(G_2,H) = \\
        \uti\big(E(G_2),E(H)\big) - \uti\big(E(G_1),E(H)\big).
        \label{eq:lem3_step2}
    \end{multline}
    Using the triangle inequality in~\eqref{eq:lem3_step2} gives
    \begin{multline}
        \uti\big(E(G_2),E(H)\big) - \uti\big(E(G_1),E(H)\big) \\ \leq d_\Delta\big(E(G_1),E(G_2)\big) = -u(G_1,G_2).
        \label{eq:lem3_step3}
    \end{multline}
    Then plugging~\eqref{eq:lem3_step2} and~\eqref{eq:lem3_step3} into~\eqref{eq:lem3_step1} gives
    \begin{equation} \label{eq:thinginlemma3}
        \Delta u \leq \max_{\substack{G_1, G_2 \in \G_n \\ \textnormal{Adj}(G_1, G_2) = 1}} -u(G_1,G_2) = A.
    \end{equation}
    Identical steps show~\eqref{eq:thinginlemma3} also holds when $u(G_2,H) \geq u(G_1,H)$.
\end{proof}

With a global sensitivity bound in place, we next develop the exponential mechanism for graphs.

\subsection{Graph Exponential Mechanism}
\label{sec:exp_mech}
The exponential mechanism formalizes the idea of generating private outputs with high utility.
We state the exponential mechanism in its standard form 
for a sensitive graph $G \in \G_n$.

\begin{definition}[{Exponential mechanism;~\cite[Chapter 3.4]{dwork2014algorithmic}}]
    Fix a privacy parameter~$\varepsilon\geq0$ and an adjacency parameter~$A \in \mathbb{N}$. 
    The exponential mechanism selects and outputs a graph~$H\in \G_n$ with probability proportional to~$\exp\big(\frac{\varepsilon u(G,H)}{2\Delta u}\big)$, where~$u$ is from Definition~\ref{def:util} and~$\Delta u$ is from Definition~\ref{def:sensitivity}.
    \label{def:exponential_mechanism}
\end{definition}

It is well-established that the exponential mechanism satisfies $\varepsilon$-differential privacy~\cite[Chapter 3.4]{dwork2014algorithmic}.
As stated, Definition~\ref{def:exponential_mechanism} requires the computation of a normalization constant
to compute probabilities of generating specific outputs. The conventional analysis
of the exponential mechanism~\cite[Chapter 3.4]{dwork2014algorithmic} 
requires 
\begin{equation} \label{eq:epsover2bound}
\frac{\exp\left(\frac{\varepsilon u(G,H)}{2\Delta u}\right)}{\exp\left(\frac{\varepsilon u(G',H)}{2\Delta u}\right)} \leq \exp\left(\frac{\varepsilon}{2}\right)
\end{equation}
for adjacent
graphs~$G$ and~$G'$. 

The reason for having~$\frac{\varepsilon}{2}$ on the right-hand side of~\eqref{eq:epsover2bound} 
is that the normalization
constant is also allowed to differ by up to a multiplicative factor of~$\frac{\varepsilon}{2}$
in the same way. 
For a graph~$G$, we can define the normalization constant
\begin{equation} \label{eq:norm_term}
C_{G} = \sum_{H\in\mathcal{G}_n}\exp\left(\frac{\varepsilon u(G,H)}{2\Delta u}\right),
\end{equation}
and for an adjacent graph~$G'$ we can define
\begin{equation} 
C_{G'} = \sum_{H\in\mathcal{G}_n}\exp\left(\frac{\varepsilon u(G',H)}{2\Delta u}\right).
\end{equation}
Then the conventional analysis of the exponential mechanism 
reserves half of the privacy budget for changes in the proportionality constant
by allowing
\begin{equation} \label{eq:norm_term2}
\frac{C_{G'}}{C_{G}} \leq \exp\left(\frac{\varepsilon}{2}\right). 
\end{equation}
Then when the input graph is~$G$ one would sample~$H \sim \frac{1}{C_G}\exp\left(\frac{\varepsilon u(G,H)}{2\Delta u}\right)$, 
and when the input graph is~$G'$ one would sample~$H \sim \frac{1}{C_{G'}}\exp\left(\frac{\varepsilon u(G',H)}{2\Delta u}\right)$. 
One easily shows that the exponential mechanism is~$\varepsilon$-differentially private
because
\begin{equation}
\frac{C_{G'}\exp\left(\frac{\varepsilon u(G,H)}{2\Delta u}\right)}{C_{G}\exp\left(\frac{\varepsilon u(G',H)}{2\Delta u}\right)} \leq \exp\left(\frac{\varepsilon}{2}\right)\exp\left(\frac{\varepsilon}{2}\right) = \exp(\varepsilon). 
\end{equation}
The bounds in terms of~$\frac{\varepsilon}{2}$ in~\eqref{eq:epsover2bound}
and~\eqref{eq:norm_term2} are why~$\frac{\varepsilon}{2}$
appears in the definition of the exponential mechanism in
Definition~\ref{def:exponential_mechanism}.

The change in the normalization constant 
can occur when the exponential mechanism is applied in a context
in which that constant depends on~$G$. 
If the normalization constant does not depend on~$G$,
then the ratio in~\eqref{eq:norm_term2} would hold with upper
bound~$\exp\left(\frac{0}{2}\right) = 1$.
In that case, one does not need to reserve half of the privacy
budget for changes in the normalization constant,
and, instead, one can use the entire privacy budget for the exponential term
in Definition~\ref{def:exponential_mechanism}. Doing so would
mean that an output~$H$ would be generated with probability proportional
to~$\exp\big(\frac{\varepsilon u(G,H)}{\Delta u}\big)$
instead of~$\exp\big(\frac{\varepsilon u(G,H)}{2\Delta u}\big)$.
In that case, the normalization constant would be~$1$ divided by a sum
of terms of the form~$\exp\big(\frac{\varepsilon u(G,H)}{\Delta u}\big)$. 

We next show that, over the space of private output graphs,
the normalization constant 
\begin{equation}\label{eq:og_c}
C=\sum_{H\in\mathcal{G}_n}\exp\left(\frac{\varepsilon u(G,H)}{\Delta u}\right)
\end{equation}
is indeed
independent of~$G$.

\begin{lemma}\label{lem:norm_const}
Fix an adjacency parameter~$A \in \mathbb{N}$
and a privacy parameter~$\varepsilon \geq 0$.
Consider~$C$ in~\eqref{eq:og_c}. 
For any sensitive graph~$G\in\mathcal{G}_n$, we have 
\begin{equation}
\begin{aligned}\label{eq:big_C}
    C &= \biggl(1+\exp\left(-\frac{\varepsilon}{A}\right) \biggr)^{\binom{n}{2}}.
\end{aligned}
\end{equation}
\end{lemma}
\begin{proof}
    The utility function $u$ depends on $G$ through the distance~$\uti(E(G), E(H))$. 
Therefore, to compute $C,$ we group all candidate private outputs in $\mathcal{G}_n$ according to their distance from $G$.
Let
\begin{equation} \label{eq:X}
\mathcal{X}(G,k)=\{H\in \G_n \ | \ \uti(E(G),E(H))=k \}
\end{equation}
be the set of candidate output graphs whose edge sets have 
\FN distance~$k$ from the edge set of~$G$.
By construction, $k\in \{0, 1, \dots, \binom{n}{2}\}$ and~$\G_n = \bigcup_{k=0}^{\binom{n}{2}}\mathcal{X}(G,k)$ for any $G\in\mathcal{G}_n$. Therefore, we express~\eqref{eq:og_c} as
\begin{equation} \label{eq:CwithX}
\begin{aligned}
    C &= \sum_{k=0}^{\binom{n}{2}} \sum_{H\in \mathcal{X}(G,k)} \exp\left(\frac{\varepsilon u(G,H)}{\Delta u}\right). \\ 
\end{aligned}
\end{equation}
For a given value of~$k$, we have $|\mathcal{X}(G,k)|=\binom{\binom{n}{2}}{k},$ and any graph $H\in\mathcal{X}(G,k)$ satisfies $u(G,H)=-k$. 
Then from~\eqref{eq:CwithX} 
\begin{equation}
\begin{aligned}
    C &= \sum_{k=0}^{\binom{n}{2}} {\binom{\binom{n}{2}}{k}} \exp\left(-\frac{k\varepsilon}{A}\right).\label{eq:mid_c}
\end{aligned}
\end{equation}
Finally, applying the binomial theorem to~\eqref{eq:mid_c} gives $$
    C = \biggl(1+\exp\left(-\frac{\varepsilon}{A}\right) 
    \biggr)^{\binom{n}{2}}.$$ 
\end{proof}

The normalization constant~$C$ in~\eqref{eq:big_C} has no dependence on $G$ because 
it is fully determined by the number of nodes in the network~$n$, the privacy parameter~$\varepsilon$, and the adjacency parameter~$A$. 
We use this normalization constant to formalize the graph exponential mechanism and state the probability of outputting any candidate graph~$H\in\mathcal{G}_n$, with no privacy budget reserved for the normalization constant. 

\begin{mechanism}[Graph Exponential Mechanism] \label{mech:em}
Fix a sensitive graph~$G \in \G_n$, an adjacency parameter~$A \in \mathbb{N}$, 
and a privacy parameter~$\varepsilon \geq 0$. 
 Then the \emph{graph exponential mechanism} $\M_{\G}$ selects and outputs a graph $H\in \G_n$ 
 with probability 
\begin{equation} \label{formula1}
    \mathbb{P}\big[\M_{\G}(G) = H\big]
    = \frac{\exp(\frac{\varepsilon u(G,H)}{A})}{\biggl(1+\exp(-\frac{\varepsilon}{A}) \biggr)^{\binom{n}{2}}}. 
\end{equation}
\end{mechanism}

The following theorem establishes that Mechanism~\ref{mech:em} is differentially private. 

\begin{theorem}\label{thm:mech1_dp}
    Fix a sensitive graph~$G \in \G_n$, an adjacency parameter~$A \in \mathbb{N}$, 
    and a privacy parameter~$\varepsilon \geq 0$. 
    Then Mechanism~\ref{mech:em} is $\varepsilon$-differentially private
    with respect to the adjacency relation in Definition~\ref{def:adj}.     
\end{theorem}
\begin{proof}
    Let $G, G'\in \G_n$ be adjacent graphs.
    For a candidate output graph~$H \in \G_n$, we analyze the ratio of the probabilities that the exponential mechanism outputs~$H$ given the inputs~$G$ and~$G'$. 
    Then
    \begin{align}
            \frac{\Prob[\M_\G (G)=H]}{\Prob[\M_\G(G')=H]}&=\frac{\exp \bigl( \frac{\varepsilon u(G,H)}{A}\bigr) / {\bigl(1+\exp(-\frac{\varepsilon}{A}) \bigr)^{\binom{n}{2}}}} {\exp \bigl( \frac{\varepsilon u(G',H)}{A}\bigr) / {\bigl(1+\exp(-\frac{\varepsilon}{A}) \bigr)^{\binom{n}{2}}}} \\
            &=\exp\biggl(\frac{\varepsilon \bigl(u(G,H)-u(G',H) \bigr)}{A}\biggr)\\
            &\leq \exp\biggl(\frac{\varepsilon \big|u(G,H)-u(G',H) \big|}{A}\biggr)\label{eq:thm1_almost}  \\
            &\leq \exp(\varepsilon) \label{eq:thm1_final},
    \end{align}
    where
    \eqref{eq:thm1_final} follows from applying Lemma~\ref{lem:sensitivity}.
\end{proof}


Although Mechanism~\ref{mech:em} has an explicit form, it may not be possible to implement it directly.
It is known that
if the set of private outputs is large,
then the exponential mechanism may be difficult to efficiently implement as written~\cite[Section 3.4]{dwork2014algorithmic}. 
To naively
implement Mechanism~\ref{mech:em}, one would assign probabilities to all $2^{\binom{n}{2}}$ graphs in $\mathcal{G}_n$,
which becomes intractable for even modest values of~$n$. 
Therefore, we next find an implementable approach for the graph exponential mechanism.

\ActivateWarningFilters[sectiontitles]
\section{Efficient Implementation Using \\ a Modified E-R Model} \label{sec:eff_mech}
\DeactivateWarningFilters[sectiontitles]
In this section we construct an efficient implementation of Mechanism~\ref{mech:em} using a modified E-R model.
For a fixed node set, the conventional E-R model~\cite{erdds1959random}
generates a random graph in which each edge 
appears independently and with the 
same probability~$p \in [0, 1]$. In this work we consider two types of edges. 
For an input graph $G \in \G_n$ and an output graph $H\in\G_n$, we want the edges in $G$ to appear in $H$ with high probability, while the edges not in $G$ should 
appear in~$H$ with low probability. 
Mathematically, we do this by using a modified E-R model with two edge probabilities of the form 
\begin{equation} \label{eq:P_eH1}
\Prob[e\in E(H)]=\begin{cases}
\begin{aligned}
    p \quad &e\in E(G)\\
    1-p \quad &e\notin E(G)\\
\end{aligned}
\end{cases}, 
\end{equation}
where $p \geq \frac{1}{2}$ to ensure that~$p \geq 1 - p$ and hence that edges
in~$G$ appear in~$H$ with a higher probability than the edges that do not appear in~$G$.
The following mechanism gives an explicit expression for $p$ in terms of the privacy parameter $\varepsilon$ and the adjacency parameter~$A$.

\begin{mechanism}[Modified E-R Model] \label{mech:2}
Fix a sensitive graph~$G \in \G_n$, an adjacency parameter~$A \in \mathbb{N}$, 
and a privacy parameter~$\varepsilon \geq 0.$
Consider a candidate output graph $H\in \G_n.$
The edge set of $H$ is constructed according to
\begin{equation} \label{eq:P_eH2}
\Prob[e\in E(H)]=\begin{cases} 
\begin{aligned}
     \frac{1}{1+\exp(-\frac{\varepsilon}{A})} \quad &e\in E(G)\\
    \frac{\exp(-\frac{\varepsilon}{A})}{1+\exp(-\frac{\varepsilon}{A})} \quad &e\notin E(G)\\
\end{aligned}
\end{cases}.
\end{equation}
\end{mechanism}

\begin{algorithm}[t]
\caption{Modified E-R Model (Mechanism~\ref{mech:2})}
\label{alg:mech2}
\begin{algorithmic}[1]
\STATEx \textbf{Input:} Sensitive graph $G=(V,E)$ with $|V|=n$, adjacency parameter $A\in\mathbb{N}$, privacy parameter $\varepsilon \geq 0$
\STATEx \textbf{Output:} Private graph $H$ on $V$
\STATE $p \gets \frac{1}{1+\exp(-\frac{\varepsilon}{A})} $
\STATE Initialize $H = \big(V, E(H)\big)$ with~$E(H) = \emptyset$
\FOR{$i=1$ \textbf{to} $n-1$}
  \FOR{$j=i+1$ \textbf{to} $n$}
    \IF{$(i,j)\in E(G)$}
      \STATE With probability $p$, add edge~$(i, j)$ to~$E(H)$
    \ELSE
      \STATE With probability $(1-p)$, add edge $(i,j)$ to $E(H)$
    \ENDIF
  \ENDFOR
\ENDFOR
\STATE \textbf{return} $H$
\end{algorithmic}
\end{algorithm}

The implementation of Mechanism~\ref{mech:2} is summarized in Algorithm~\ref{alg:mech2}.
We use~$\mathcal{M}_E(G)$ to denote a private graph generated from a sensitive graph $G$ using Mechanism~\ref{mech:2}. 
Next, we establish that Mechanism~\ref{mech:em} and Mechanism~\ref{mech:2} 
in fact define the same distribution over private graphs. 

\begin{theorem}\label{thm:equiv}
    Fix a sensitive graph~$G \in \G_n$. Then 
    Mechanism \ref{mech:em} and Mechanism \ref{mech:2} are equivalent for any privacy parameter $\varepsilon \geq 0$ and adjacency parameter $A \in \mathbb{N}$
    in the sense that they define the same distribution over~$\mathcal{G}_n$. 
\end{theorem}
    
\begin{proof}    
    Given $G \in \G_n$ and $H \in \G_n$ with ${u(G,H)=-k}$, we analyze the quantities
    \begin{equation}
        \Prob_k^\mathcal{G} = \mathbb{P}[\mathcal{M}_\mathcal{G}(G)=H \ s.t. \ u(G,H)=-k]\label{eq:mech1_initP}
    \end{equation}
    and
    \begin{equation}
        \Prob_k^E = \mathbb{P}[\mathcal{M}_E(G)=H \ s.t. \ u(G,H)=-k] \label{eq:mech2_initP}
    \end{equation}
    for~$k \in \mathbb{N}$. 
    Here, $\Prob_k^\mathcal{G}$ and $\Prob_k^E$ denote the probabilities that the outputs of Mechanisms~\ref{mech:em} and~\ref{mech:2} have utility ${-k}$,  respectively. 
    To establish the equivalence of Mechanism~\ref{mech:em} and Mechanism~\ref{mech:2},
    we show that $\Prob_k^\mathcal{G}=\Prob_k^E$ for an arbitrary~$k \in \mathbb{N}$. 
    
    For Mechanism~\ref{mech:em} we have 
    \begin{align}
        \Prob_k^\mathcal{G}&=\sum_{H\in\mathcal{X}(G,k)}\mathbb{P}\big[\M_{\G}(G) = H\big]\\
        &=\frac{{\binom{\binom{n}{2}}{k}} \exp(-\frac{\varepsilon k}{A})}{\biggl(1+\exp(-\frac{\varepsilon}{A})\biggr)^{\binom{n}{2}}} \label{eq:P_k1},
    \end{align}
    where $\mathcal{X}(G,k)$ is from~\eqref{eq:X}. 
    Equation~\eqref{eq:P_k1} follows using~$|\mathcal{X}(G,k)|=\binom{\binom{n}{2}}{k}$,  
    substituting the probability of outputting $H\in\mathcal{X}(G,k)$ from~\eqref{formula1}, and using $u(G,H) = -k$. 
    
    For Mechanism~\ref{mech:2}, we start from the general form~\eqref{eq:P_eH1}.
    Suppose that~$G$ and~$H$ differ in~$k$ edges, where~$H$ contains~$m$
    edges that are not in~$G$, $H$ is missing~$\ell$ edges that are in~$G$,
    and~$k = \ell + m$. 
    
    The probability of adding $m$ edges to 
    $E(H)$ that were not in 
    $E(G)$ is 
    \begin{equation}
       { \binom{|E(\bar{G})|}{m} p^{|E(\bar{G})|-m} (1-p)^{m},}
    \end{equation}
    where $\bar{G}$ is the complementary graph of $G$ and $|E(\bar{G})| = \binom{n}{2}-|E(G)|.$
    The probability that~$E(H)$ is missing~$\ell$ edges that were in~$E(G)$ is 
    \begin{equation}
    {\binom{|E(G)|}{\ell} p^{|E(G)|-\ell} (1-p)^{\ell}.}
    \end{equation}
    Then summing over every possible combination of~$k = \ell + m$ edge additions and edge deletions, $\mathbb{P}_k^E$ is given by
    \begin{equation}
        \mathbb{P}_k^E = \sum_{\ell=0}^k \binom{|E(G)|}{\ell} \binom{|E(\bar{G})|}{k-\ell} p^{|E(G)|+|E(\bar{G})|-k} (1-p)^{k} .
    \end{equation}
    Substituting in~$k = \ell+m$ and~$|E(G)| + |E(\bar{G})| = \binom{n}{2}$ 
    gives 
    \begin{align} 
         \mathbb{P}_k^E &= p^{{\binom{n}{2}}-k} (1-p)^k \sum_{\ell=0}^k \binom{|E(G)|}{\ell} \binom{|E(\bar{G}|}{m}\\
         & = p^{{\binom{n}{2}}-k} (1-p)^k  {\binom{\binom{n}{2}}{k}},\label{eq:P_k2}
    \end{align}
    where~\eqref{eq:P_k2} follows from the Chu-Vandermonde identity~\cite[Exercise 3.2(a)]{koepf2014}.    

    To derive an explicit expression for $p$ that sets~$\Prob_k^\mathcal{G}=\Prob_k^E$,
    we set~\eqref{eq:P_k1} and~\eqref{eq:P_k2} equal to each other to find
    \begin{equation}
        \frac{{\binom{\binom{n}{2}}{k}} \exp(- \frac{\varepsilon k}{A})}{\biggl(1+\exp(-\frac{\varepsilon}{A})\biggr)^{\binom{n}{2}}} = p^{{\binom{n}{2}}-k} (1-p)^k {\binom{\binom{n}{2}}{k}}.
    \end{equation}
    Dividing by ${\binom{\binom{n}{2}}{k}}$ and taking logs of both sides, we reach
    \begin{multline}
        -{\binom{n}{2}} \log\biggl(1+\exp\left(-\frac{\varepsilon}{A}\right) \biggr)-\frac{k\varepsilon}{A} \\
        = \left({\binom{n}{2}}-k\right) \log(p) + k\log(1-p), 
    \end{multline}
    which holds for any~$k \in \{0, \ldots, \binom{n}{2}\}$. For simplicity, we set~$k=0$ to obtain
    \begin{equation}
        -{\binom{n}{2}}\log\biggl(1+\exp\left(-\frac{\varepsilon}{A}\right)\biggr) =  {\binom{n}{2}}\log(p),
    \end{equation}
    which gives
    \begin{equation}
        p = \frac{1}{1+\exp(-\frac{\varepsilon}{A})},
    \end{equation}
    which is the same form for~$p$ presented in \eqref{eq:P_eH2}. Then Mechanism \ref{mech:em} and Mechanism \ref{mech:2} are equivalent for any privacy parameter $\varepsilon \ge 0$ and adjacency parameter $A\in \mathbb{N}$. 
    It can be verified that this equivalence holds all~$k \in \{0, \ldots, \binom{n}{2}\}$.
\end{proof}

Theorem~\ref{thm:equiv} shows that Mechanisms~\ref{mech:em} and~\ref{mech:2} are equivalent. As demonstrated in Theorem~\ref{thm:mech1_dp}, Mechanism~\ref{mech:em} enforces~$\varepsilon$-differential privacy. 
Therefore, these theorems together imply that graphs generated with Mechanism~\ref{mech:2} also keep
a sensitive graph $\varepsilon$-differentially private.

We observe that the limiting behavior of~$\varepsilon$ is intuitive. 
First recall that privacy is strengthened when~$\varepsilon$ is smaller. 
For~$\varepsilon = \infty$, privacy is absent and~$p=1$, 
which means that Mechanisms~\ref{mech:em}
and~\ref{mech:2} both simply output their sensitive input graphs with no modifications. 
Conversely, for~$\varepsilon = 0$, any two adjacent input graphs must produce private output
graphs with identical statistics. This condition is indeed enforced because~$p = \frac{1}{2}$
when~$\varepsilon = 0$, which indicates that each privacy mechanism ignores the sensitive
input graph and samples an output graph uniformly from the space of output graphs. 

Algorithm~\ref{alg:mech2} (which implements Mechanism~\ref{mech:2}) constructs $H$ by performing a biased coin flip for each possible edge.
After computing the probabilities in~\eqref{eq:P_eH2} in constant time, the algorithm executes $\binom{n}{2}$ constant-time steps.
Therefore the time complexity of Algorithm~\ref{alg:mech2} is $\Theta(n^2).$
In contrast, a naive implementation of Mechanism~\ref{mech:em} 
has complexity~$\Theta(2^{\binom{n}{2}})$, which makes such an implementation infeasible for even moderate graph sizes.
Overall, Mechanism~\ref{mech:2} reduces the complexity of sampling 
private output graphs
from~$\Theta(2^{\binom{n}{2}})$ to 
$\Theta(n^2)$.

\section{Results for a Real-World Complex Network} \label{sec:IV}
In this section, we present numerical simulations in which we privately compute the Laplacian spectrum 
for a graph that models a Facebook user's network of friends. Data for these simulations is available at~\cite{snapnets}.  We compare the accuracy of our proposed mechanism with the accuracy of the mechanism from~\cite{hawkins2023node} that was designed to privatize Laplacian spectra and is the current
state of the art for doing so. 
This section solves Problem~\ref{prob:3}.

\subsection{Privacy Mechanism Setup}
First, we define the graph Laplacian and its spectrum.
Given a graph $G \in \G_n$, let $d_i=|\{j\in V \ | \ (i,j)\in E\}|$ denote the degree of node $i\in V$. The degree matrix $D(G)\in \mathbb{R}^{n\times n}$ is the diagonal matrix $D(G)=\textnormal{diag}(d_1,\dots,d_n).$ The adjacency matrix of $G$ is 
\begin{equation}
    \big(Z(G)\big)_{ij}= \begin{cases}
        1 \quad \textnormal{if} \ (i,j)\in E \\
        0 \quad \textnormal{otherwise}
    \end{cases}.
\end{equation}

The Laplacian of a graph~$G$ is defined as~$L(G)=D(G)-Z(G)$, where~$L$ is symmetric and hence has real eigenvalues. 
Let the eigenvalues of~$L$ be ordered according to~$\lambda_1 \leq \lambda_2 \leq \cdots \leq \lambda_n$. The matrix $L$ is positive semi-definite, and thus~$\lambda_i \geq 0$ for all~$i\in\{1,2,\dots, n\}$.
Furthermore, every undirected graph $G$ has $\lambda_1 = 0$, which makes its value non-sensitive because
it is graph-independent. We therefore seek to provide privacy to $\{\lambda_2, \dots, \lambda_n\}$. 

The bounded Laplace mechanism from~\cite{hawkins2023node} privatizes~$\{\lambda_2,\dots,\lambda_n\}$ by adding bounded Laplace noise independently to each eigenvalue.
Privately computing these eigenvalues requires
first computing each eigenvalue and then adding noise to it, which amounts to conducting~$(n-1)$ queries of the underlying sensitive graph $G$.
Lemma~\ref{lemma:comp} shows that if the bounded Laplace mechanism provides $\lambda_i$ with $\varepsilon_{bl}$-differential privacy
for each~$i \in \{2, \ldots, n\}$, 
then the underlying graph is protected by $(n-1)\varepsilon_{bl}$-differential privacy overall. 

On the other hand, Mechanism~\ref{mech:2} generates an entire private graph~$\mathcal{M}_E(G)$ while guaranteeing $\varepsilon$-differential privacy. Since differential privacy is immune to post-processing according to Lemma~\ref{lemma:immunity}, any property/properties of $\mathcal{M}_E(G)$ can be computed
and $\varepsilon$-differential privacy is preserved.
To implement Mechanism~\ref{mech:2}, we synthesize a private graph $\mathcal{M}_E(G)$ with $\varepsilon$-differential privacy, and compute $\{\lambda_2,\dots,\lambda_n\}$ using the Laplacian of $\mathcal{M}_E(G)$, which is simply post-processing. 

For Mechanism~\ref{mech:2} and the bounded Laplace mechanism from~\cite{hawkins2023node} to provide equivalent levels of privacy to 
the underlying graph~$G$, they must satisfy
\begin{equation}
    \varepsilon=(n-1)\varepsilon_{bl}.\label{eq:fair_eps}
\end{equation}
Calibrating both mechanisms according to~\eqref{eq:fair_eps} provides a fair comparison, 
and we now compare the accuracy of these mechanisms using real data.
For both mechanisms, 
we denote privatized eigenvalues by~$\{\tilde\lambda_2,\dots,\tilde\lambda_n\}.$

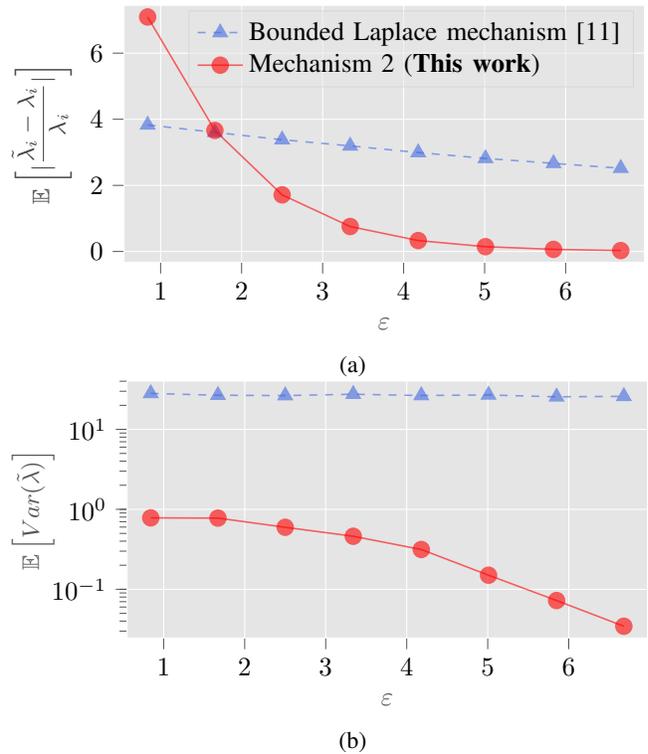
\begin{figure}
\setlength{\figH}{5cm}
    \setlength{\figW}{8.5cm}
    \captionsetup[subfigure]{justification=centering}
    \centering
    \begin{subfigure}{.4\textwidth}
        \centering
        \hspace*{-1.5cm}
        \definecolor{dimgray85}{RGB}{85,85,85}
\definecolor{gainsboro229}{RGB}{229,229,229}
\definecolor{lightgray204}{RGB}{204,204,204}
\definecolor{royalblue}{RGB}{65,105,225}
\begin{tikzpicture}

\begin{axis}[%
axis background/.style={fill=gainsboro229},
axis line style={white},
height=\figH,
legend cell align={left},
legend style={fill opacity=0.5, draw opacity=0.5, text opacity=1, draw=lightgray204, fill=gainsboro229},
tick align=outside,
tick pos=left,
width=\figW,
x grid style={white},
xmajorgrids,
xmin=0.548, xmax=6.972,
xtick style={color=dimgray85},
xlabel=\textcolor{dimgray85}{\(\displaystyle \varepsilon\)},
y grid style={white},
ylabel=\textcolor{darkgray}{\(\displaystyle \mathbb{E}\left[ |\frac{\tilde{\lambda}_i-\lambda_i}{\lambda_i} |\right]\)},
ymajorgrids,
ymin=-0.325504989468746, ymax=7.44686416660689,
ytick style={color=dimgray85},
ylabel style={font=\small},
]
\addplot [semithick, royalblue, opacity=0.6, dashed, mark=triangle*, mark size=3, mark options={solid}]
table {%
0.84 3.82900220222397
1.67 3.60374999620344
2.5 3.3809986705849
3.34 3.19442036033197
4.18 2.99280832605078
5.01 2.81577534187749
5.85 2.66589898259348
6.68 2.51891761199688
};
\addlegendentry{Bounded Laplace mechanism~\cite{hawkins2023node}}
\addplot [semithick, red, opacity=0.6, mark=*, mark size=3, mark options={solid}]
table {%
0.84 7.09387195424129
1.67 3.66208877966479
2.5 1.7129677017537
3.34 0.755293215168772
4.18 0.330414357463731
5.01 0.145727527361709
5.85 0.0641638730601437
6.68 0.0277845176256011
};
\addlegendentry{Mechanism~\ref{mech:2} (\textbf{This work})}
\end{axis}
\end{tikzpicture}%
        \caption{}
        \label{fig:expn}
    \end{subfigure}
    \begin{subfigure}{0.4\textwidth}
        \centering
        \hspace*{-1.5cm}
        \definecolor{dimgray85}{RGB}{85,85,85}
\definecolor{gainsboro229}{RGB}{229,229,229}
\definecolor{lightgray204}{RGB}{204,204,204}
\definecolor{royalblue}{RGB}{65,105,225}
\begin{tikzpicture}

\begin{axis}[%
axis background/.style={fill=gainsboro229},
axis line style={white},
height=\figH,
log basis y={10},
tick align=outside,
tick pos=left,
width=\figW,
x grid style={white},
xlabel=\textcolor{dimgray85}{\(\displaystyle \varepsilon\)},
xmajorgrids,
xmin=0.548, xmax=6.972,
xtick style={font=\tiny, color=dimgray85},
y grid style={white},
ylabel=\textcolor{dimgray85}{\(\displaystyle \mathbb{E}\left[Var(\tilde\lambda)\right]\)},
ylabel style={font=\small},
ymajorgrids,
ymin=0.0246516325901209, ymax=40.2889835901693,
ymode=log,
ytick style={color=dimgray85},
ylabel style={font=\small, yshift=-8},
yshift=-1cm
]
\addplot [semithick, royalblue, opacity=0.6, dashed, mark=triangle*, mark size=3, mark options={solid}]
table {%
0.84 28.1589098489275
1.67 26.7189335706497
2.5 26.4538950820796
3.34 27.5150916777204
4.18 26.5120647858941
5.01 26.9312314262206
5.85 25.5128348834257
6.68 25.8890252076501
};
\addplot [semithick, red, opacity=0.6, mark=*, mark size=3, mark options={solid}]
table {%
0.84 0.78200464241898
1.67 0.776927207731
2.5 0.596403945096723
3.34 0.460508221777889
4.18 0.314510501078831
5.01 0.150040208598064
5.85 0.0723548144207777
6.68 0.0345069539145854
};
\end{axis}
\end{tikzpicture}%
        \caption{}
        \label{fig:var}
    \end{subfigure}
    \caption{The (a) empirical mean and (b) empirical variance of the error in~$\{\tilde\lambda_2,\dots,\tilde\lambda_{168}\}$ for privacy parameters 
    $\varepsilon \in \big\{\varepsilon_l \mid \varepsilon_l = 0.835l, \, l \in \{1,2, \dots, 8\}\big\}$.}
    \label{fig:ExpnVar}
\end{figure} 

\subsection{Experiments with a Real-World Complex Network}
Let~$G\in \G_{168}$ be the graph from~\cite{snapnets} with user ID 686. This graph corresponds to a specific Facebook user's network
of friends; see~\cite{snapnets} for more details.
The graph is undirected, has~$n=168$ nodes, and~$|E(G)|=1,656$ edges. We fix the adjacency parameter $A=1$ and sweep over multiple values of $\varepsilon$.
Specifically, for Mechanism~\ref{mech:2} we use $\varepsilon \in \big\{ \varepsilon_l \mid \  \varepsilon_l = 0.835l, \, l \in \{1,2, \dots, 8\} \big\}$ and for the bounded Laplace mechanism we use $\varepsilon_{bl} = \varepsilon/(n-1)$ according to~\eqref{eq:fair_eps}.
For each $\varepsilon$, we generate~${M=1,000}$ private spectra~$\{\tilde\lambda_2,\dots,\tilde\lambda_n\}$
with Mechanism~\ref{mech:2}, and we do the same for the bounded Laplace mechanism from~\cite{hawkins2023node}. 
To measure accuracy we use each mechanism to compute $\frac{1}{n-1}\sum_{i=2}^n |\frac{\tilde{\lambda}_i-\lambda_i}{\lambda_i}|$ for each private spectrum
that we generate, and then we compute the average of this quantity over all $M$ spectra. These results are presented in Figure~\ref{fig:expn}
for a range of values of~$\varepsilon$. 
We also measure the empirical variance. 
For each~$i \in [168] \backslash \{1\}$, 
let $Var(\tilde{\lambda_i})$ be the variance 
of the private values of~$\lambda_i$ 
computed over $M=1,000$ samples. 
The value  $\mathbb{E}\left[Var(\tilde\lambda)\right]=\frac{1}{n-1}\sum_{i=2}^{n}Var(\tilde\lambda_i)$ is shown in Figure~\ref{fig:var}
for a range of values of~$\varepsilon$.

Figure~\ref{fig:expn} shows that when~$\varepsilon>1.67$, the proposed mechanism achieves lower error than the bounded Laplace mechanism. 
Figure~\ref{fig:var} shows that, for all 
values of $\varepsilon$ tested,
the average variance $\mathbb{E}[Var(\tilde{\lambda}_i)]$ is more than an order of magnitude lower
for Mechanism~\ref{mech:2} than it is for the bounded Laplace mechanism from~\cite{hawkins2023node}. 
The accuracy of our proposed mechanism is only worse than the bounded Laplace mechanism for the 
smallest values of~$\varepsilon$ tested.
In Figure~\ref{fig:expn}, when~$\varepsilon=2.50,$  the proposed mechanism achieves $49.34\%$ less error than the bounded Laplace mechanism, which
demonstrates the substantially improved accuracy of our mechanism for typical levels of privacy.
Indeed, $\varepsilon = 2.50$ is within conventionally desirable ranges~\cite{hsu2014differential} for~$\varepsilon$, and the error we incur
is approximately half of that incurred by the state-of-the-art mechanism in~\cite{hawkins2023node}.
Mechanism~\ref{mech:2} therefore attains strong privacy and high accuracy, as desired. 


\section{Conclusion}\label{sec:V}
This paper presented a differentially private mechanism for privatizing the edge sets of undirected complex networks. We first defined a utility function and then proposed the graph exponential mechanism. To make it computationally efficient, we developed a modified E-R model to independently add each edge
to a private output graph. 
Numerical results showed that the proposed mechanism is more accurate 
than the state of the art
for computing the private spectra of graph Laplacians. 
Future work includes generalizing the proposed mechanism to directed and weighted graphs, and using it to compute other algebraic graph properties.

\balance

\bibliography{root}
\bibliographystyle{ieeetr}

\end{document}